\documentclass{article}

\usepackage{fullpage}
\usepackage{amsmath,amsfonts,amssymb,amsthm,thmtools,hyperref,cleveref,enumerate,color}
\usepackage[most]{tcolorbox}

\newtheorem{theorem}{Theorem}[section]
\newtheorem{proposition}[theorem]{Proposition}
\newtheorem{lemma}[theorem]{Lemma}
\newtheorem{corollary}[theorem]{Corollary}
\theoremstyle{definition}
\newtheorem{definition}[theorem]{Definition}

\providecommand{\sqbinom}[2]{\genfrac{[}{]}{0pt}{}{#1}{#2}}
\providecommand{\FF}{\mathbb{F}}
\providecommand{\RR}{\mathbb{R}}
\DeclareMathOperator{\codim}{codim}
\providecommand{\subs}{\triangleleft}

\providecommand{\blueemail}[1]{\href{mailto:#1}{\textcolor{blue}{\texttt{#1}}}}
\providecommand{\Cset}{C}
\providecommand{\Nmain}{N_{\ref{thm:main}}}
\providecommand{\Nlist}{N_{\ref{pro:step1}}}
\providecommand{\Nhomo}{N_{\ref{lem:homogeneous}}}
\providecommand{\Nramsey}{N_{\ref{lem:ramsey-induction}}}
\providecommand{\Nramseyc}{N_{\ref{lem:ramsey-conclusion}}}
\providecommand{\Nramaff}{N_{\ref{pro:ramsey-affine}}}
\providecommand{\Nrampro}{N_{\ref{pro:ramsey-projective}}}
\providecommand{\Nramprohelp}{N_{\ref{lem:ramsey-projective-helper}}}
\providecommand{\NHJ}{N_{\ref{pro:HJ}}}
\providecommand{\Alpha}{\mathrm{A}}
\providecommand{\hatAlpha}{\hat{\Alpha}}

\title{Boolean degree one functions on the Grassmann scheme}
\author{Yuval Filmus\footnote{Taub Faculty of Computer Science and Faculty of Mathematics, Technion Israel Institute of Technology, \blueemail{yuvalfi@cs.technion.ac.il}. Research funded by ISF grant 507/24.}}

\begin{document}

\maketitle
\begin{abstract}
Ferdinand Ihringer proved that Boolean degree one functions on the Grassmann scheme $J_q(n,k)$ are trivial when $\min(k,n-k) \ge 2$ and $n$ is large enough. We provide a mostly self-contained exposition of this result. All results were formalized by ChatGPT~5.6~Sol, and are available on GitHub at \href{https://github.com/YuvalFilmus/grassmann-degree-one}{\texttt{YuvalFilmus/grassmann-degree-one}}.
\end{abstract}

\section{Introduction}
\label{sec:introduction}

The Grassmann scheme $J_q(n,k)$ consists of all $k$-dimensional subspaces of $\FF_q^n$. A function $f\colon J_q(n,k) \to \RR$ is \emph{Boolean} if $f(V) \in \{0,1\}$ for all $V \in J_q(n,k)$. It has \emph{degree one} if it can be written as
\[
 f = \sum_p c_p x_p,
\]
where the sum goes over all points (one-dimensional subspaces of $\FF_q^n$), and $x_p \in \{0,1\}$ is the indicator of the input containing $p$. We could also allow a constant term, but this is unnecessary since
\[
 \sum_p x_p = \sqbinom{k}{1}_q
\]
over $J_q(n,k)$. (On the right-hand side we have a $q$-binomial coefficient, which in this case equals $q^{k-1} + \cdots + 1$.)

\begin{tcolorbox}[colback=blue!5!white,colframe=blue!75!black]
  Which degree one functions on $\FF_q^n$ are Boolean?
\end{tcolorbox}

The obvious examples are $0, 1, x_p, 1-x_p$. A less obvious example is the function $y_r \in \{0,1\}$, which is the indicator of the input being \emph{orthogonal} to $r$:
\[
 y_r = 1 - \left(\sqbinom{k}{1}_q - \sqbinom{k-1}{1}_q\right)^{-1} \sum_{p \not\perp r} x_p.
\]

Incidentally, this shows that the property of having degree one is preserved by the duality of $J_q(n,k)$ and $J_q(n,n-k)$. Given a function $f$ on $J_q(n,k)$, we can define its dual $f^\perp$ on $J_q(n,n-k)$ by $f^\perp(V) = f(V^\perp)$. Since $x_p^\perp = y_p$, it follows that $f$ has degree one iff $f^\perp$ has degree one.

We can also combine $x_p$ and $y_r$ if they do not ``conflict''. This leads to the following definition.
\begin{definition} \label{def:trivial}
A Boolean degree one function on $J_q(n,k)$ is \emph{trivial} if it belongs to the following list:
\[
 0, 1, x_p, 1-x_p, y_r, 1-y_r, x_p+y_r, 1-x_p-y_r,
\]
where in the latter two cases, $p \not\perp r$.

We say that $J_q(n,k)$ is \emph{degree-one trivial (DOT)} if all Boolean degree one functions on $J_q(n,k)$ are trivial.
\end{definition}

If $k = 1$ then every function has degree one, and so in general $J_q(n,k)$ is not DOT (unless $n$ is very small). By duality, the same holds when $k = n-1$. When $q=2$, these are the only problematic cases: Filmus and Ihringer~\cite{FI19} showed that $J_q(n,k)$ is DOT whenever $\min(k,n-k) \ge 2$. 

Surprisingly, when $q > 2$, there are Boolean degree one functions on $J_q(4,2)$ which are not trivial. Here is an example for $q = 3$, due to Bruen and Drudge~\cite[Corollary 2.1]{BD99}.\footnote{Bruen and Drudge's construction is parametrized by an elliptic quadric $\mathcal{O}$. In this case, $\mathcal{O} = Q^{-1}(0)$. We thank ChatGPT~5.5 for this explicit instantiation.} For a point $p = \langle(x,y,z,w)\rangle$, let
\[
 Q(p) = x^2 + y^2 + z^2 - w^2,
\]
which is well-defined since $Q(x,y,z,w) = Q(2x,2y,2z,2w)$. The counterexample is
\[
 f = \frac{1}{2} \sum_{Q(p) = 0} x_p + \frac{1}{6} \sum_{Q(p) = 1} x_p - \frac{1}{6} \sum_{Q(p) = 2} x_p.
\]
To check that $f$ is Boolean, observe that for any line (two-dimensional subspace) $\ell$, the reduced quadratic $Q|_\ell$ is equivalent to one of $xy,x^2,2x^2,x^2+y^2$. The corresponding weight distributions (number of points on $\ell$ on which $Q$ takes each value $0,1,2$) are $(2,1,1),(1,3,0),(1,0,3),(0,2,2)$, and the corresponding values of $f$ are $1,1,0,0$.

\begin{tcolorbox}[colback=blue!5!white,colframe=blue!75!black]
  When is $J_q(n,k)$ degree-one trivial?
\end{tcolorbox}

Drudge~\cite[Theorem 6.4]{DrudgePhD} showed that $J_3(n,2)$ is DOT for all $n \ge 5$, and this was extended to $J_4(n,2)$ and $J_5(n,2)$ by Gavrilyuk and Mogilnykh~\cite{GM14} and by Gavrilyuk and Matkin~\cite{GM18}, respectively. Filmus and Ihringer~\cite{FI19} extended this to $J_3(n,k),J_4(n,k),J_5(n,k)$ whenever $\min(k,n-k) \ge 2$ and $n \ge 5$. They did so using an inductive argument showing that if $J_q(n,k)$ is DOT then so is $J_q(n+1,k+1)$. 
What prevented them from proving a similar result for other values of $q$ is a missing base case. This base case was finally provided by Ihringer~\cite{Ihr24}.

\begin{theorem}[Ihringer] \label{thm:main}
For every prime power $q$ there exists a constant $\Nmain(q)$ such that if $n \ge \Nmain(q)$ and $\min(k,n-k) \ge 2$ then $J_q(n,k)$ is DOT.
\end{theorem}

In this paper we give an exposition of Ihringer's proof which is nearly self-contained. We also replace the inductive argument of Filmus and Ihringer by a more direct one.

The optimal value $\Nmain(q)$ is not known. When $q = 2$, the optimal value is $\Nmain(2) = 4$, a result which we also reproduce. Ihringer's proof gives a value of $\Nmain(q)$ which has a tower dependence on $q$, resulting from the application of Ramsey theorems. However, at the moment we cannot rule out that $\Nmain(q) = 5$ works for all $q$.

\subsection{Background}
\label{sec:background}

The study of Boolean degree one functions on the Grassmann scheme was initiated in a paper of Cameron and Liebler~\cite{CL82} on $J_q(4,2)$, and for this reason such functions are known as \emph{Cameron--Liebler line classes}. A blog post by Ferdinand Ihringer~\cite{IhringerBlog} describes many other interpretations of the same concept.

Boolean degree one functions on various domains also arise naturally in Erd\H{o}s--Ko--Rado (EKR) theory. The \emph{Johnson scheme} $J(n,k)$ is the set of all subsets of $[n] := \{1,\dots,n\}$ of size $k$. The classical EKR theorem~\cite{EKR61} states that if $\mathcal{F} \subseteq J(n,k)$ is \emph{intersecting} (any two sets in $\mathcal{F}$ intersect) and $k \leq n/2$, then $|\mathcal{F}| \leq \binom{n-1}{k-1}$. This is known as the \emph{EKR property} for $J(n,k)$. If furthermore $k < n/2$, then the only families achieving the bound are the \emph{stars} $\mathcal{F}_i = \{ S : i \in S \}$, which is known as the \emph{strict EKR property}.

There are many ways to prove the EKR theorem. Lov\'asz~\cite{Lovasz79} gave a spectral proof based on his celebrated theta function. When $k < n/2$, the proof shows that every intersecting family of maximum size is in the linear span of the stars (identifying a family with its characteristic function), which is known as the \emph{EKR module property}. Equivalently, every such intersecting family is a Boolean degree one function. Since the only degree one functions on the Johnson scheme (when $\min(k,n-k) \ge 2$) are $0, 1, x_i, 1-x_i$, it follows that such intersecting families are stars.

The EKR theorem has been extended to various other domains, including the Grassmann scheme~\cite{FW86}. In many cases, including that of the Grassmann scheme, the only known proof is spectral, and the proof implies the EKR module property; see the monograph~\cite{GM16} for a thorough study. Understanding the extremal families thus reduces to understanding Boolean degree one functions on these various domains. Godsil and Meagher~\cite{GM09} developed a specialized technique to this end, which only applies to intersecting families.

In the case of $t$-intersecting families (every two objects in the family have at least $t$ points in common), the spectral approach shows that extremal families are Boolean degree $t$ functions. Boolean degree $t$ functions over the Boolean cube $\{0,1\}^t$ have a special structure: they depend on $O(2^t)$ coordinates~\cite{CHS20,Wellens22}, and they can be computed by decision trees of depth $O(t^3)$~\cite{Mid04}. This follows from a well-developed theory of complexity measures~\cite{BdW02}, which has been extended to other domains~\cite{DFLLV21}. Unfortunately, the relevant techniques do not seem to extend to the Grassmann scheme.

\subsection{Proof overview}
\label{sec:proof-overview}

The proof consists of two parts. The first part shows that $J_q(n,2)$ is DOT for all $n \ge \Nlist(q)$, and the second part deduces \Cref{thm:main}.

The first part itself is composed of two steps. The first step shows that, after possibly replacing $f$ by $1-f$, the coefficients $c_p$ in the representation $f = \sum_p c_p x_p$ belong to the set
\[
 -\frac{q-1}{q}, 0, \frac{1}{q}, 1.
\]
The argument, which uses Ramsey theory, appears in \Cref{sec:coeffs}.

The second step considers the patterns of coefficients on lines and planes to deduce the DOT property, and appears in \Cref{sec:lines}.

The second part derives the full \Cref{thm:main}. To explain the basic idea, we switch notation: $J_q(a;b)$ is the same as $J_q(a+b,b)$. We show that if $J_q(a;b)$ is DOT then $J_q(a';b)$ is also DOT for all $a' \ge a$ by considering restrictions of $J_q(a';b)$ to copies of $J_q(a;b)$. Dually, if $J_q(a';b)$ is DOT then $J_q(a';b')$ is also DOT for all $b' \ge b$, completing the proof of \Cref{thm:main}. This argument appears in \Cref{sec:restriction}.

The proof in the second part, while simple in principle, involves some case analysis. The reader might prefer reading first the easier case of the Johnson scheme (where the DOT property is also easy to prove directly), which appears in \Cref{sec:restriction-johnson}.

When $q = 2$, \Cref{thm:main} holds for $\Nmain(2) = 4$. We prove this in \Cref{sec:F2}.

The proof of \Cref{thm:main} uses two Ramesey theorems, which we deduce in~\Cref{sec:ramsey} from the Hales--Jewett theorem for completeness.

\section{Isolating the coefficients}
\label{sec:coeffs}

The goal of this section is to prove the following result.

\begin{proposition} \label{pro:step1}
For every prime power $q$ there exists $\Nlist(q)$ such that the following holds for all $n \ge \Nlist(q)$.

If $f\colon J_q(n,2) \to \RR$ is a Boolean degree one function then either
\[
 f = \sum_p c_p x_p, \text{ where } c_p \in \left\{-\frac{q-1}{q}, 0, \frac{1}{q}, 1 \right\} \text{ for all } p,
\]
or the same holds for $1 - f$.
\end{proposition}

The proof proceeds in four steps:
\begin{enumerate}
\item By considering the restrictions of $f$ to planes (three-dimensional subspaces), we show that the coefficients $c_p$ are \emph{quantized}, that is, belong to some finite set which depends only on $q$.
\item Using the Ramsey theorem for vector spaces, we find a large ``reference'' subspace $U$ which is homogeneous: the coefficients of all $p \in U$ are the same. The common coefficient must be either $0$ or $\frac{1}{q+1}$. 
\item For every point $p \notin U$, we use repeated applications of the Ramsey theorem for affine vector spaces to find a subspace of $U$ which is ``homogeneous with respect to $p$''.
\item We deduce that $c_p$ has one of the claimed values.
\end{enumerate}
The most difficult step is the third one.


\bigskip

We start by showing that the coefficients $c_p$ are quantized.

\begin{lemma} \label{lem:quantization}
There exists a finite set $\Cset(q)$ such that if $n \ge 3$ then $c_p \in \Cset(q)$ for all $p$.
\end{lemma}
\begin{proof}
Every function on $J_q(3,2)$ has degree one. Indeed, if $g\colon J_q(3,2) \to \RR$ then $g^\perp\colon J_q(3,1) \to \RR$. Since $x_p$ is a delta function on $J_q(3,1)$,
\[
 g^\perp = \sum_{p \in J_q(3,1)} g^\perp(p) x_p,
\]
and so
\[
 g = \sum_{p \in J_q(3,1)} g(p^\perp) y_p.
\]

Every function $g\colon J_q(3,2) \to \RR$ can thus be written in the form
\[
 g = \sum_{p \in J_q(3,2)} d_p x_p.
\]
The space of real-valued functions on $J_q(3,2)$ has dimension $\sqbinom{3}{2}_q$ and is spanned by the $\sqbinom{3}{1}_q$ functions $x_p$. Since $\sqbinom{3}{2}_q = \sqbinom{3}{1}_q$, we conclude that the functions $x_p$ form a basis, and so the representation is unique.

There are finitely many Boolean functions on $J_q(3,2)$, and consequently, the coefficients $d_p$ occurring in such functions belong to some finite set $\Cset(q)$.

We conclude the proof by showing that $c_p \in \Cset(q)$ for all $p$. Every $p$ is contained in some plane $A$. The restriction of $f$ to $A$ can be written as
\[
 f|_A = \sum_{p \in A} c_p x_p.
\]
Since $f|_A$ is Boolean and the set of functions on $A$ have the same structure as the set of functions on $J_q(3,2)$, we deduce that $c_p \in \Cset(q)$, as required.
\end{proof}

Next, we find a large homogeneous subspace.

\begin{lemma} \label{lem:homogeneous}
For every $m \ge 2$, provided $n \ge \Nhomo(q,m)$, there is an $m$-dimensional subspace $U \subseteq \FF_q^n$ such that $c_p = c$ for all $p \in U$, where $c \in \{0, \frac{1}{q+1}\}$.
\end{lemma}
\begin{proof}
In view of \Cref{lem:quantization} and for large enough $n$, the Ramsey theorem for vector spaces~\cite{GLR72,Spencer79} implies the existence of a homogeneous $m$-dimensional subspace $U \subseteq \FF_q^n$, say all coefficients are equal to $c$. Since $m \ge 2$, the subspace $U$ contains a line $\ell$. Since
\[
 f(\ell) = \sum_{p \in \ell} c_p = \sqbinom{2}{1}_q c = (q + 1) c,
\]
we conclude that either $c = 0$ (if $f(\ell) = 0$) or $c = \frac{1}{q+1}$ (if $f(\ell) = 1$).
\end{proof}

If $p \in U$ then $c_p = c$. Otherwise, we find a subspace of $U$ which is ``homogeneous with respect to $p$''. We color each vector $u \in U$ by $c_{\langle u + p' \rangle}$, where $p'$ is an arbitrary non-zero vector in $p$. We color vectors $u \in U$ rather than points $\langle u \rangle$ since $c_{\langle \alpha u + p' \rangle}$ could depend on $\alpha$. The relevant Ramsey theorem is then the Ramsey theorem for affine vector spaces~\cite{GLR72}, which states that some large affine subspace $X \subseteq U$ is homogeneous: all vectors in $X$ have the same color $d(X)$.

We can write $X = W + u$ for some $u \in U$, possibly $u = 0$. If $w \in W$ is non-zero then $X$ contains $\beta w + u$ for all $\beta \in \mathbb{F}_q$, and so the coefficients of the points $\langle \beta w + u + p' \rangle$ are all colored by $c(X)$. Together with $\langle w \rangle \in U$ they form a line, and so $qc(X) + c \in \{0, 1\}$, restricting the possible values of $c(X)$.

If $X$ were a subspace, then we could consider the line connecting $p$ and $\langle w \rangle$. This line contains the additional points $\langle \beta w + p' \rangle$ for $\beta \ne 0$, and so we obtain a formula for $c_p$ in terms of $c(X)$. However, we are not guaranteed that $X$ is a subspace. Instead, we use repeated applications of the Ramsey theorem to find $W,u$ such that the subspaces $W + \alpha u + p'$ are homogeneous for all non-zero $\alpha$.

\begin{lemma} \label{lem:ramsey-induction}
For every $r \ge 1$ and $\Alpha \subseteq \FF_q^\times$, the following holds for the subspace $U$ constructed in \Cref{lem:homogeneous}, provided $m \ge \Nramsey(q,r,\Alpha)$.

For every point $\langle p' \rangle \notin U$ there exist an $r$-dimensional subspace $W \subseteq U$ and a vector $u \in U \setminus W$ such that $W + \alpha u + p'$ is homogeneous for all $\alpha \in \Alpha$: there exists $c(\alpha) \in \{\frac{-c}{q}, \frac{1-c}{q} \}$ such that $c_{\langle w + \alpha u + p' \rangle} = c(\alpha)$ for all $w \in W$.
\end{lemma}
\begin{proof}
The proof is by induction on $\Alpha$.

\paragraph{Base case} Suppose that $\Alpha = \emptyset$. In this case, we just need to find an $r$-dimensional subspace $W \subseteq U$ and some $u \in U \setminus W$. We choose $m = r + 1$, let $u$ be an arbitrary non-zero vector in $U$, and let $W$ be an arbitrary $r$-dimensional subspace of $U$ not containing $u$.

\paragraph{Induction step} Suppose that $\Alpha$ is non-empty, say $\alpha_* \in \Alpha$. We apply the induction hypothesis with an appropriate $\hat r$ and $\hat \Alpha = \Alpha \setminus \{\alpha_*\}$, obtaining an $\hat r$-dimensional subspace $\hat W \subseteq U$ and a vector $\hat u \in U \setminus \hat W$ such that $\hat W + \alpha \hat u + p'$ is homogeneous for all $\alpha \in \hat \Alpha$.

Color the vectors in $\hat W$ as follows: $\hat w \in \hat W$ gets the color $c_{\langle \hat w + \alpha_* \hat u + p' \rangle}$. In view of \Cref{lem:quantization} and for an appropriate choice of $\hat r$, the Ramsey theorem for affine vector spaces~\cite{GLR72} applied to $\hat W$ implies the existence of an $r$-dimensional affine subspace $W + z \subseteq \hat{W}$ such that $W + z + \alpha_* \hat u + p'$ is homogeneous. Note that $W \subseteq \hat{W}$ and $z \in \hat{W}$.

Let $u = \alpha_*^{-1} z + \hat u \in U$. Note that $u \notin W$, since otherwise $\hat{u} \in \hat{W}$. To complete the proof, we show that $W + \alpha u + p'$ is homogeneous for all $\alpha \in \Alpha$ and that the common value of all coefficients is in $\{0, \frac{1}{q}\}$.

By construction, $W + \alpha_* u + p' = W + z + \alpha_* \hat u + p'$ is homogeneous.

If $\alpha \in \hatAlpha$ then by assumption $\hat{W} + \alpha \hat{u} + p'$ is homogeneous, and so $W + \alpha u + p'$ is also homogeneous, since $W + \alpha u = W + \alpha \alpha_*^{-1} z + \alpha \hat u \subseteq \hat{W} + \alpha \hat u$.

For each $\alpha \in \Alpha$, choose any non-zero $w \in W$. Then $c_{\langle \beta w + \alpha u + p' \rangle} = c(\alpha)$ for all $\beta \in \mathbb{F}_q$ and some $c(\alpha) \in \Cset(q)$. The line $\ell$ connecting $\langle w \rangle$ and $\langle \alpha u + p' \rangle$ thus contains one point $\langle w \rangle$ whose coefficient is $c$, and $q$ points $\langle \beta w + \alpha u + p' \rangle$ whose coefficients are $c(\alpha)$.\footnote{To see that this is indeed a line we need to check that $w, \alpha u + p' \neq 0$ and that $\langle w \rangle \neq \langle \alpha u + p' \rangle$. We chose $w$ to be non-zero. We have $\alpha u + p' \neq 0$ since $\alpha u + p' \notin U$. Since $w \in U$, it follows that $\langle w \rangle \neq \langle \alpha u + p' \rangle$. In the sequel we omit such routine verification.} The sum of these coefficients is $f(\ell) \in \{0,1\}$, and so $c(\alpha) \in \{\frac{-c}{q}, \frac{1-c}{q}\}$.
\end{proof}

One more application of the Ramsey theorem for affine vector spaces allows us to extend \Cref{lem:ramsey-induction} to all $\alpha$, at the cost of having $W$ be an affine subspace. We will actually need a slightly stronger statement.

\begin{lemma} \label{lem:ramsey-conclusion}
For every $r \ge 1$, the following holds for the subspace $U$ constructed in \Cref{lem:homogeneous}, provided $m \ge \Nramseyc(q,r)$.

For every point $\langle p' \rangle$ there exist an $r$-dimensional affine subspace $W + t \subseteq U$ and a vector $u \neq 0$ such that
\begin{enumerate}[(a)]
\item $W + t + p'$ is homogeneous: there exists $c(0) \in \{\frac{-c}{q},\frac{1-c}{q}\}$ such that $c_{\langle w + t + p' \rangle} = c(0)$ for all $w \in W$.
\item $\langle W,t \rangle + \alpha u + p'$ is homogeneous for all $\alpha \neq 0$: there exists $c(\alpha) \in \{\frac{-c}{q},\frac{1-c}{q}\}$ such that $c_{\langle z + p' \rangle} = c(\alpha)$ for all $z \in \langle W,t \rangle$.
\end{enumerate}
\end{lemma}
\begin{proof}
For a given point $\langle p' \rangle \notin U$, apply \Cref{lem:ramsey-induction} with an appropriate $\hat r$ and $\Alpha = \FF_q^\times$ to obtain an $\hat r$-dimensional subspace $\hat W \subseteq U$ and a vector $u \in U \setminus \hat W$ such that $\hat W + \alpha u + p'$ is homogeneous for all $\alpha \neq 0$, with common value $c(\alpha) \in \{\frac{-c}{q}, \frac{1-c}{q}\}$. In particular, $u \neq 0$.

Color each vector $\hat w \in \hat W$ with $c_{\langle \hat w + p' \rangle}$. In view of \Cref{lem:quantization} and for an appropriate choice of $\hat r$, the Ramsey theorem for affine vector spaces~\cite{GLR72} applied to $\hat W$ implies the existence of an $r$-dimensional affine subspace $W + t \subseteq \hat W$ such that $W + t + p'$ is homogeneous, with some common value $c(0)$. Since $\langle W,t \rangle \subseteq \hat W$, we see that $\langle W,t \rangle + \alpha u + p'$ is homogeneous for all $\alpha \neq 0$, with some common value $c(\alpha) \in \{\frac{-c}{q}, \frac{1-c}{q}\}$.

Since $r \ge 1$, we can find a point $\langle s \rangle \in W$. The line $\ell$ connecting $\langle s \rangle$ and $\langle t + p' \rangle$ thus contains one point $\langle s \rangle$ whose coefficient is $c$, and $q$ points $\langle \beta s + t + p' \rangle$ whose common value is $c(0)$. The sum of these coefficients is $f(\ell) \in \{0,1\}$, and so $c(0) \in \{\frac{-c}{q}, \frac{1-c}{q}\}$.
\end{proof}

At this point we can conclude the proof of \Cref{pro:step1}.

\begin{proof}[Proof of \Cref{pro:step1}]
Apply \Cref{lem:ramsey-conclusion} with $r := 1$ in conjunction with \Cref{lem:homogeneous} to deduce that for $\Nlist(q) := \Nramseyc(q,1)$, the following holds:
\begin{enumerate}[(a)]
\item There exists a subspace $U$ such that $c_p = c$ for all $p \in U$, where $c \in \{0, \frac{1}{q+1}\}$.
\item For every $\langle p' \rangle \notin U$ there exists a one-dimensional affine subspace $\langle s \rangle + t$, a vector $u \neq 0$, and a function $c\colon \FF_q \to \{\frac{-c}{q}, \frac{1-c}{q}\}$ such that $c_{\langle x + p' \rangle} = c(0)$ for all $x \in \langle s \rangle + t$ and $c_{\langle y + \alpha u + p' \rangle} = c(\alpha)$ for all $y \in \langle s,t \rangle$ and $\alpha \neq 0$.
\end{enumerate}

We will show that if $c = 0$ then $c_p \in \{-\frac{q-1}{q},0,\frac{1}{q},1\}$ for all points $p$. If $c = \frac{1}{q+1}$ then we observe that
\[
 1 - f = \sum_p \left(\frac{1}{q+1} - c_p\right) x_p,
\]
and so the same will hold for $1 - f$, since the $c$ corresponding to $1 - f$ (for the same subspace $U$) is $0$.

\smallskip

If $p \in U$ then $c_p = 0$, so suppose that $p \notin U$. Choose some vector $p'$ such that $p = \langle p' \rangle$.

Consider the line $\ell_1$ connecting the points $\langle u \rangle$ and $\langle t+p' \rangle$. This line contains the point $\langle u \rangle \in U$ and the points $\langle t+\alpha u + p' \rangle$ for all $\alpha$. The corresponding coefficients are $c = 0$ and $c(\alpha)$. Since $f(\ell_1) \in \{0, 1\}$, it follows that $\sum_{\alpha \in \FF_q} c(\alpha) \in \{0, 1\}$.

Since $c(\alpha) \in \{0,\frac 1q\}$ for all $\alpha$, this leaves two possibilities: $c(\alpha) = 0$ for all $\alpha$, or $c(\alpha) = \frac1q$ for all $\alpha$. Denote this common value by $d \in \{0, \frac1q\}$.

Consider next the line $\ell_2$ connecting the points $\langle t+u \rangle$ and $\langle p' \rangle$. This line contains the points $\langle p' \rangle$, $\langle t+u \rangle \in U$, and $\langle \alpha t + \alpha u + p' \rangle$ for all $\alpha \neq 0$. The corresponding coefficients are $c_p$, $c = 0$, and $c(\alpha)$ for $\alpha \neq 0$. Since $f(\ell_2) \in \{0,1\}$, it follows that
\[
 c_p = f(\ell_2) - \sum_{\alpha \neq 0} c(\alpha) = f(\ell_2) - (q-1)d \in \{0,1\} - \{0,\tfrac{q-1}{q}\} = \{-\tfrac{q-1}{q},0,\tfrac{1}{q},1\}. \qedhere
\]
\end{proof}

\section{Main theorem for lines}
\label{sec:lines}

In this section, we show that functions satisfying the conclusion of \Cref{pro:step1} are trivial.

\begin{proposition} \label{pro:step2}
Suppose that $n \ge 3$ and that $f\colon J_q(n,2) \to \RR$ is Boolean and has the form
\[
 f = \sum_p c_p x_p, \text{ where } c_p \in \left\{-\frac{q-1}{q}, 0, \frac{1}{q}, 1 \right\} \text{ for all } p.
\]

Then $f$ has one of the following forms:
\[
 0, x_p, 1 - y_r, 1 - x_p - y_r,
\]
where in the last case $p \not\perp r$.
\end{proposition}

The first step is determining the possible sets of coefficients along a line.

\begin{lemma} \label{lem:line-coeffs}
For any line $\ell$ in $J_q(n,2)$, the $q+1$ coefficients $c_p$ for $p \in \ell$ are equal to one of the following multisets:
\[
 \{0^{q+1}\}, \{1,0^q\}, \{0,\tfrac{1}{q}^q\}, \{0,-\tfrac{q-1}{q},\tfrac{1}{q}^{q-1}\}, \{1,-\tfrac{q-1}{q},\tfrac{1}{q}^{q-1}\}, \{1,-\tfrac{q-1}{q}^2,\tfrac{1}{q}^{q-2}\}.
\]

In particular, if $Z = \{0,1\}$ and $Q = \{-\frac{q-1}{q},\frac{1}{q}\}$, then either all coefficients belong to $Z$, or exactly one coefficient belongs to $Z$.
\end{lemma}
\begin{proof}
Suppose that the multiset is $\{-\frac{q-1}{q}^a,\frac{1}{q}^b,1^c,0^d\}$, where $a+b+c+d = q+1$. Since $f(\ell) = \sum_{p \in \ell} c_p$, we have
\[
 -(q-1)a + b + qc \in \{0,q\}.
\]

Considering this modulo $q$, we see that $a + b \equiv 0 \pmod{q}$, and so either $a = b = 0$ or $a + b = q$.

If $a = b = 0$ then we have $qc \in \{0,q\}$. Consequently $c \in \{0,1\}$, and we obtain $\{0^{q+1}\}$ or $\{1,0^q\}$.

If $a + b = q$ then $c \in \{0,1\}$. Substituting $b = q - a$ gives
\[
 q(1-a+c) \in \{0,q\}.
\]

If $c = 0$ then $a \in \{0,1\}$, and we obtain $\{0,\frac{1}{q}^q\}$ or $\{0,-\frac{q-1}{q},\frac{1}{q}^{q-1}\}$.

If $c = 1$ then $a \in \{1,2\}$, and we obtain $\{1,-\tfrac{q-1}{q},\tfrac{1}{q}^{q-1}\}$ or $\{1,-\tfrac{q-1}{q}^2,\tfrac{1}{q}^{q-2}\}$.
\end{proof}

In the list provided by the \namecref{lem:line-coeffs}, the latter two possibilities cannot actually occur, but we can only rule them out by considering planes.

The key idea in the proof of \Cref{pro:step2} is the observation that the set of points whose color lies in $Z$ is a subspace, since a line contains either $1$ or $q+1$ coefficients from $Z$. Since this subspace intersects every line, it must have codimension at most~$1$. The rest of the proof is a short case analysis.

\begin{proof}[Proof of \Cref{pro:step2}]
Let $V = \{ p : c_p \in Z \}$. We claim that $V$ is a (projective) subspace.\footnote{Equivalently, $\bigcup_{c_p \in Z} p$ is a linear subspace.} Indeed, if $p_1,p_2 \in V$ then applying \Cref{lem:line-coeffs}, we see that all coefficients in the line connecting $p_1$ and $p_2$ belong to $V$. Moreover, \Cref{lem:line-coeffs} makes clear that $V$ is non-empty.

\Cref{lem:line-coeffs} implies that every line must intersect $V$, and so $\codim(V) \leq 1$.

The easier case is when $\codim(V) = 0$, that is, all coefficients belong to $Z$. If $c_{p_1} = c_{p_2} = 1$ for two different points, then applying \Cref{lem:line-coeffs} to the line connecting them, we obtain a contradiction. Hence at most one coefficient is equal to $1$. If all coefficients are equal to $0$ then $f = 0$. If only $c_p = 1$, then $f = x_p$.

The more complicated case is when $\codim(V) = 1$. We claim that no coefficient can equal $1$. Indeed, suppose that $c_p = 1$; the argument of the preceding paragraph shows that this is the unique point whose coefficient is $1$. There are $\sqbinom{n-1}{1}_q - \sqbinom{n-2}{1}_q = q^{n-2} \ge 2$ lines through $p$ which are not contained in $V$. In particular, we can find two different such lines $\ell_1,\ell_2$.

Each line $\ell_i$ contains one point whose coefficient is $c_p = 1$ and one point whose coefficient is in $Q$, hence \Cref{lem:line-coeffs} shows that some point $r_i \in \ell_i$ has coefficient $c_{r_i} = -\frac{q-1}{q}$. Since $\ell_1,\ell_2$ intersect only at $p$, the two points $r_1,r_2$ are different. The line $\ell_3$ connecting them has two coefficients equal to $-\frac{q-1}{q}$, hence according to \Cref{lem:line-coeffs}, it must contain a point whose coefficient is $1$, which must be $p$. However, in that case $\ell_3$ contains two points of $\ell_1$ and two points of $\ell_2$, hence it must coincide with both of them, contradicting $\ell_1 \neq \ell_2$.

\Cref{lem:line-coeffs} thus implies that the multiset of coefficients in each line is one of the following:
\[
 \{0^{q+1}\}, \{0,\tfrac{1}{q}^q\}, \{0,-\tfrac{q-1}{q},\tfrac{1}{q}^{q-1}\}.
\]
In particular, at most one point $p$ has coefficient $c_p = -\frac{q-1}{q}$, since if there were two such points, then we would obtain a contradiction by considering the line connecting them.

Since $\codim(V) = 1$, we can write $V = r^\perp$ for some point $r$. If no point has coefficient $-\frac{q-1}{q}$ then
\[
 f = \frac{1}{q} \sum_{p \not\perp r} x_p = 1 - y_r.
\]

If $c_{p_0} = -\frac{q-1}{q}=\frac1q-1$, where $p_0 \not\perp r$, then
\[
 f = \frac{1}{q} \sum_{p \not\perp r} x_p - x_{p_0} = 1 - x_{p_0} - y_r. \qedhere
\]
\end{proof}

\section{Restriction}
\label{sec:restriction}

In this section we complete the proof of \Cref{thm:main} by proving the following result, where we remind the reader that $J_q(a;b)$ stands for $J_q(a+b,a)$.

\begin{proposition} \label{pro:step3}
If $J_q(a;b)$ is DOT and $a,b \ge 2$ then $J_q(a';b)$ is DOT for all $a' \ge a$.    
\end{proposition}

Before proving this \namecref{pro:step3}, let us show how it implies \Cref{thm:main}. The crucial observation is that the DOT property is preserved by duality.

\begin{lemma} \label{lem:DOT-duality}
If $J_q(a;b)$ is DOT then so is $J_q(b;a)$.
\end{lemma}
\begin{proof}
Suppose that $f\colon J_q(b;a) \to \RR$ is a degree one Boolean function. The function $f^\perp\colon J_q(a;b) \to \RR$ is also a Boolean degree one function. By assumption, $f^\perp$ is trivial. Since $x_p^\perp = y_p$ and $y_p^\perp = x_p$, it follows that $f$ is trivial as well.
\end{proof}

\begin{corollary} \label{cor:step3}
If $J_q(a;b)$ is DOT and $a,b \ge 2$ then $J_q(a';b')$ is DOT for all $a' \ge a$ and $b' \geq b$.
\end{corollary}
\begin{proof}
Since $J_q(a;b)$ is DOT, \Cref{pro:step3} shows that $J_q(a';b)$ is DOT. \Cref{lem:DOT-duality} implies that $J_q(b;a')$ is DOT. \Cref{pro:step3} shows that $J_q(b';a')$ is DOT, and \Cref{lem:DOT-duality} completes the proof, showing that $J_q(a';b')$ is DOT.
\end{proof}

\Cref{thm:main} immediately follows.

\begin{proof}[Proof of \Cref{thm:main}]
\Cref{pro:step1,pro:step2} together show that $J_q(a;2)$ is DOT for all $a \ge \Nlist(q)-2$.
Defining $N'_0(q) = \max(\Nlist(q),4)$, \Cref{cor:step3} shows that $J_q(a';b')$ is DOT whenever $a' \ge N'_0(q)-2$ and $b' \ge 2$, and \Cref{lem:DOT-duality} implies that the same holds for $J_q(b';a')$.

Let $\Nmain(q) = 2(N'_0(q)-2)$. Defining $a' = \max(k,n-k)$ and $b' = \min(k,n-k)$, the assumptions imply that $a' \ge N'_0(q)-2$ and $b' \ge 2$, and so $J_q(n,k)$ is DOT. 
\end{proof}

\medskip

The main idea of the proof of \Cref{pro:step3} is to consider restrictions of the Boolean degree one function $f$ on $J_q(a';b)$ to subspaces of dimension $a+b$. For each such subspace $V$, the restriction $f|_V$ is a Boolean degree one function on $J_q(a;b)$, which is trivial by assumption. 

The Grassmann graph $G_q(a'+b,a+b)$ is the graph on $J_q(a'+b,a+b)$ in which two subspaces are connected if their intersection has dimension $a+b-1$. If $V \sim V'$ then the form of $f|_V$ significantly restricts the possible forms of $f|_{V'}$; we will show this by considering the possible forms of $f|_{V \cap V'}$. Using the connectivity of the Grassmann graph, this will allow us to show that $f$ is trivial.

The proof involves a fair amount of case analysis. For this reason, we work out the easier case of the Johnson scheme in \Cref{sec:restriction-johnson}. The Johnson scheme $J(n,k)$ consists of all $k$-subsets of $\{1,\dots,n\}$. The role of $x_p$ is played by $x_i$, which indicates that the input subset contains~$i$. The analog of \Cref{thm:main} for the Johnson scheme is also easy to prove directly; it holds whenever $\min(k,n-k) \ge 2$.

In order to cut the number of cases, we employ the following convention: $g^+ = g$ and $g^- = 1-g$. Every trivial function thus has one of the forms $0^\pm,x_p^\pm,y_r^\pm,(x_p+y_r)^\pm$.

\begin{proof}[Proof of \Cref{pro:step3}]
Let $n = a+b$ and $n' = a'+b$; we can assume that $n' > n$. Let $f$ be a Boolean degree one function on $J_q(n',b)$. For every $n$-dimensional subspace $V$, the restriction $f|_V$ is a Boolean degree one function on $J_q(n,b)$, and so by assumption, $f|_V$ is trivial.

The proof will involve the Grassmann graph $G = G_q(n',n)$ on the vertex set $J_q(n',n)$, where we connect two subspaces $V,V'$ if their intersection has dimension $n-1$. It is well-known that this graph is connected. If $V,V'$ are neighbors we denote this by $V \sim V'$.
In this case $V \cap V'$ is a codimension~$1$ subspace of $V$, which we denote by $V \cap V' \subs V$.

The subgraph $G_{r^-}$ of $G$ consisting of all subspaces not orthogonal to $r$ is also connected (when $n' > n$). To see this, let $V,V' \in G_q(n',n)$ be two subspaces not orthogonal to $r$. Let $p \in V$ and $p' \in V'$ be arbitrary vectors which are not orthogonal to $r$. Since $n \ge 2$, we can find $W \in G_q(n',n)$ containing both $p$ and $p'$. The subgraph $G_p$ of $G$ consisting of all subspaces containing $p$ is a Grassmann graph and so it is connected. Therefore we can find a path in $G_p$ from $V$ to $W$. We can similarly find a path in $G_{p'}$ from $W$ to $V'$. All vertices on the corresponding path from $V$ to $V'$ are subspaces containing either $p$ or $p'$, and so are not orthogonal to $r$.

\smallskip

The first step in the proof involves constructing the ``restriction table'', which describes the possible structures of $f|_U$ given $f|_V$, where $U \subs V \in J_q(n',n)$:
\[
\begin{array}{c|lll}
f|_V & f|_U \\\hline
0^\pm & 0^\pm \\
x_p^\pm & 0^\pm \text{ if } p \notin U, & x_p^\pm \text{ if } p \in U \\
y_r^\pm & 0^\mp \text{ if } r \perp U, & y_r^\pm \text{ if } r \not\perp U \\
(x_p+y_r)^\pm & 0^\mp \text{ if } r \perp U, & (x_p+y_r)^\pm \text{ if } p \in U, & y_r^\pm \text{ otherwise} 
\end{array}
\]


Let us say that a trivial function is of \emph{positive type} if it is of the form $0,x_p,y_r^-,(x_p+y_r)^-$. Otherwise, it is of \emph{negative type}. If $V \sim V'$ then $V \cap V' \subs V,V'$, and so the restriction table implies that $f|_V$ and $f|_{V'}$ have the same type. Since the Grassmann graph is connected, we see that either all $f|_V$ are of positive type, or all $f|_V$ are of negative type. Without loss of generality, assume that all $f|_V$ are of positive type.

In order to cut the case analysis,\footnote{We thank ChatGPT~5.6~Sol for this idea.} we write each function of positive type as $\alpha(f|_V) + \beta(f|_V)$, where $\alpha(f|_V)$ is either $0$ or of the form $\pm x_p$, and $\beta(f|_V)$ is either $0$ or of the form $y_r^-$. The corresponding restriction tables (derived from the previous restriction table) are

\[
\begin{array}[t]{c|l}
\alpha(f|_V) & \alpha(f|_U) \\\hline
0 & 0 \\
\pm x_p & 0 \text{ if } p \in U, \\ & \pm x_p \text{ if } p \notin U
\end{array}
\quad
\begin{array}[t]{c|l}
\beta(f|_V) & \beta(f|_U) \\\hline
0 & 0 \\
y_r^- & 0 \text{ if } r \perp U, \\ & y_r^- \text{ if } r \not\perp U
\end{array}
\]

We determine the possible behaviors of $\alpha(f|_V)$ and $\beta(f|_V)$ separately.

If $\alpha(f|_{V_0}) = \pm x_p$ for some $V_0,p$ then $\alpha(f|_V) = \pm x_p$ whenever $p \in V$ (with the same sign).
Indeed, consider the subgraph $G_p$ of $G$ consisting of all subspaces containing $p$, which is isomorphic to $G_q(n'-1,n-1)$, and is therefore connected. The restriction table implies that for $V \stackrel{G_p}\sim V'$, if $\alpha(f|_V) = \pm x_p$ then $\alpha(f|_{V'}) = \pm x_p$ as well. Since $\alpha(f|_{V_0}) = \pm x_p$ and $G_p$ is connected, it follows that $\alpha(f|_V) = \pm x_p$ for all $V \in G_p$.

We claim that moreover, $\alpha(f|_V) = 0$ whenever $p \notin V$. To see this, suppose that $\alpha(f|_V) \in \{ x_{p'}, -x_{p'} \}$. Since $n \ge 2$, we can find $W \in J_q(n',n)$ containing both $p$ and $p'$; but then $\alpha(f|_W) = \pm x_p$ and $\alpha(f|_W) \in \{x_{p'}, -x_{p'}\}$, and we reach a contradiction.

Similarly, if $\beta(f|_{V_0}) = y_r^-$ for some $V_0,r$ then $\beta(f|_V) = y_r^-$ whenever $r \not\perp V$.
Indeed, consider the subgraph $G_{r^-}$ of $G$ consisting of all subspaces not perpendicular to $r$, which is connected as shown above. The restriction table implies that for $V \stackrel{G_{r^-}}\sim V'$, if $\beta(f|_V) = y_r^-$ then $\beta(f|_{V'}) = y_r^-$ as well. Since $\beta(f|_{V_0}) = y_r^-$ and $G_{r^-}$ is connected, it follows that $\beta(f|_V) = y_r^-$ for all $V \in G_{r^-}$.

We claim that moreover, $\beta(f|_V) = 0$ whenever $r \perp V$. To see this, suppose that $\beta(f|_V) = y_{r'}^-$. Since $n \ge 2$, we can find a point $p \not\perp r,r'$ and a subspace $W \in J_q(n',n)$ containing $p$; but then $\beta(f|_W)$ would be equal to both $y_r^-$ and $y_{r'}^-$, and we reach a contradiction.

Considering the possible ways in which $\alpha(f|_V)$ and $\beta(f|_V)$ combine to make $f|_V$, we see that one of the following must hold:
\begin{itemize}
\item $\alpha(f|_V) = \beta(f|_V) = 0$ for all $V$. In this case, $f = 0$.
\item $\alpha(f|_V) \in \{0, x_p\}$, according to whether $p \in V$, and $\beta(f|_V) = 0$ for all $V$. In this case, $f = x_p$.
\item $\alpha(f|_V) = 0$ for all $V$, and $\beta(f|_V) \in \{0, y_r^-\}$, according to whether $r \perp V$. In this case, $f = y_r^-$.
\item $\alpha(f|_V) \in \{0, -x_p\}$, according to whether $p \in V$ or not, and $\beta(f|_V) \in \{0, y_r^-\}$, according to whether $r \perp V$ or not. Since $r \perp V$ must imply $p \notin V$, necessarily $p \not\perp r$. In this case, $f = (x_p + y_r)^-$. \qedhere
\end{itemize}
\end{proof}

\section{Field of size two}
\label{sec:F2}

In this section we show that we can take $\Nmain(2) = 4$ in \Cref{thm:main}. In view of \Cref{cor:step3}, it suffices to prove the following.

\begin{proposition} \label{pro:F2}
The Grassmann scheme $J_2(4,2)$ is DOT.
\end{proposition}

The proof of \Cref{lem:quantization} shows that every function on $J_q(3,2)$ has degree one, and furthermore, every such function has a unique representation as a degree one polynomial. In the case of $q = 2$, we can determine this representation explicitly for all Boolean functions.

\begin{lemma} \label{lem:J232}
Let $f\colon J_2(3,2) \to \RR$ be a Boolean function.

If $f$ is trivial, then the coefficients $c_p$ belong to the set
\[
 \left\{
 -\frac{2}{3},-\frac{1}{2},-\frac{1}{6},0,\frac{1}{3},\frac{1}{2},\frac{5}{6},1
 \right\}.
\]

Otherwise, there exists a line $\ell$ and a point $r \in \ell$ such that
\begin{align*}
f &= \frac{2}{3} \sum_{\substack{s \in \ell \\ s \neq r}} x_s - \frac{1}{3} x_r + \frac{1}{6} \sum_{t \notin \ell} x_t, \text{ or } \\
f &= -\frac{1}{3} \sum_{\substack{s \in \ell \\ s \neq r}} x_s + \frac{2}{3} x_r + \frac{1}{6} \sum_{t \notin \ell} x_t.
\end{align*}
\end{lemma}
\begin{proof}
The proof of \Cref{pro:step2} gives the expansions of the functions $0,x_p,1-y_r,1-x_p-y_r$. These involve the coefficients $-\frac{1}{2},0,\frac{1}{2},1$. The other trivial functions are obtained by complementing these functions. On $J_2(3,2)$ we have $1 = \frac{1}{3} \sum_p x_p$, and so the corresponding coefficients are obtained from the previous list by the transformation $c \mapsto \frac{1}{3} - c$.

The other two listed functions are also Boolean. To see this, note that one of the following holds for every line: (i) the line equals $\ell$; (ii) the line intersects $\ell$ at $r$; (iii) the line intersects $\ell$ at a point other than $r$. The corresponding values of $f$ are $1,0,1$ in the first case, and $0,1,0$ in the second case.

Since $\sqbinom{3}{2}_2 = 7$, there are $2^7 = 128$ Boolean functions on $J_2(3,2)$. Of those:
\begin{itemize}
\item $1$ each have the form $0, 1$.
\item $7$ each have the form $x_p,1-x_p,y_r,1-y_r$.
\item $7 \times 4$ each have the form $x_p+y_r,1-x_p-y_r$ (where $p \notin r^\perp$).
\item $7 \times 3$ each have one of the two non-trivial forms.
\end{itemize}
These are all the functions since
\[
 2 \cdot 1 + 4 \cdot 7 + 2 \cdot 28 + 2 \cdot 21 = 128. \qedhere
\]
\end{proof}

Considering restrictions to planes allows us to prove \Cref{pro:F2}.

\begin{proof}[Proof of \Cref{pro:F2}]
Let $f\colon J_2(4,2) \to \RR$ be a Boolean degree one function, and write
\[
 f = \sum_p c_p x_p.
\]

For each plane $A$, the restriction of $f$ to $A$ has the form
\[
 f|_A = \sum_{p \in A} c_p x_p.
\]

In view of \Cref{lem:J232}, either all coefficients in $A$ belong to $C_+ \cup C_-$, or all of them belong to $D$, where
\[
 C_+ = \{-\tfrac12,0,\tfrac12,1\}, C_- = \{-\tfrac23,-\tfrac16,\tfrac13,\tfrac56\}, D = \{-\tfrac13,\tfrac16,\tfrac23\}.
\]

We can classify the planes in $J_2(4,2)$ into type $C$ and type $D$ accordingly. We claim that all planes need to have the same type. Indeed, a plane of type $C$ and a plane of type $D$ have a line in common, and the corresponding coefficients lie in both $C_+ \cup C_-$ and $D$, which is impossible.

Let us first rule out the case in which all planes are of type $D$. Indeed, if this is the case, then $\frac47$ of the coefficients must be $\frac16$, since each plane contains exactly $4$ coefficients equal to $\frac16$. However, this is impossible, since there are exactly $\sqbinom{4}{1}_2 = 15$ coefficients, and $7 \nmid 15$.

Hence all planes are of type $C$. We claim that for each line, either all coefficients belong to $C_+$, or all of them belong to $C_-$. Indeed, $C_+ \subseteq \frac{\mathbb{Z}}{2}$ and $C_- \subseteq \frac{1}{3} + \frac{\mathbb{Z}}{2}$. If we add three such coefficients and get something in $\{0,1\} \subseteq \frac{\mathbb{Z}}{2}$, then necessarily either all such coefficients come from $C_+$, or all of them come from $C_-$.\footnote{This argument, and the one in the following paragraph, works for all $q$.}

We can classify the lines into two types $C_+,C_-$ accordingly. Intersecting lines have the same type. Since the Grassmann graph on the lines is connected, we see that all lines have the same type. In other words, either all coefficients are in $C_+$, or all of them are in $C_-$.

If all coefficients are in $C_+$ then \Cref{pro:step2} completes the proof. If all coefficients are in $C_-$ then the coefficients of $1-f$ are in $C_+$, and again \Cref{pro:step2} completes the proof.
\end{proof}

\appendix

\section{Restriction on the Johnson scheme}
\label{sec:restriction-johnson}

The Johnson scheme $J(n,k)$ consists of all subsets of $[n] := \{1,\dots,n\}$ of size $k$. A function on $J(n,k)$ has \emph{degree one} if it can be written as
\[
 f = \sum_{i=1}^n c_i x_i,
\]
where $x_i \in \{0,1\}$ indicates whether the input subset contains $i$.
We could also allow a constant term, but this is not necessary since $\sum_i x_i = k$ on $J(n,k)$.

The analog of \Cref{thm:main} for the Johnson scheme is as follows.

\begin{theorem} \label{thm:main-johnson}
If $\min(k,n-k) \ge 2$ then every Boolean degree one function on $J(n,k)$ is of one of the forms $0, 1, x_i, 1-x_i$.    
\end{theorem}
\begin{proof}
Let $f = \sum_{i=1}^n c_i x_i$ be a Boolean function on $J(n,k)$.

For every distinct $i,j$ we can find a set $S \in J(n,k)$ containing $i$ but not $j$. Then $f(S) - f(S \triangle \{i,j\}) = c_i - c_j$, showing that $c_i - c_j \in \{-1,0,1\}$. If $c = \min_i c_i$, then this shows that $c_i \in \{c, c+1\}$ for all $i$. Accordingly, we define $\delta_i := c_i - c \in \{0,1\}$.

For every distinct $i,j,u,v$ we can find a set $S \in J(n,k)$ containing $i,j$ but not $u,v$. Then $f(S) - f(S \triangle \{i,j,u,v\}) = c_i + c_j - c_u - c_v = \delta_i + \delta_j - \delta_u - \delta_v$. Since the right-hand side belongs to $\{-1,0,1\}$, this rules out the possibility that $\delta_i=\delta_j=1$ and $\delta_u=\delta_v=0$. Therefore either at most one $\delta_i$ is equal to $1$, or at most one $\delta_i$ is equal to $0$ (and so exactly one $\delta_i$ is equal to $0$, by definition of $c$).

If $\delta_i = 0$ for all $i$ then $f$ is constant, and so either $f = 0$ or $f = 1$.

If $\delta_{i_0} = 1$ and $\delta_i = 0$ for $i \neq i_0$ then choose a set $S_0$ not containing $i_0$ and a set $S_1$ containing it. Since $f(S_0) = kc$ and $f(S_1) = kc+1$, we see that $c = 0$, and so $f = x_{i_0}$.

If $\delta_{i_0} = 0$ and $\delta_i = 1$ for $i \neq i_0$ then choose a set $S_0$ not containing $i_0$ and a set $S_1$ containing it. Since $f(S_0) = k(c+1)$ and $f(S_1) = k(c+1)-1$, we see that $c+1 = 1/k$, and so $f = 1 - x_{i_0}$.
\end{proof}

While \Cref{thm:main-johnson} is easy to prove directly, we describe in this section an alternative proof based on an analog of \Cref{pro:step3}.

\begin{proposition} \label{pro:step3-johnson}
Let $n > k \ge 2$. Suppose that every Boolean degree one function on $J(n,k)$ is of one of the forms $0, 1, x_i, 1-x_i$. Then the same holds for $J(n',k)$ for all $n' \ge n$.
\end{proposition}

Given this \namecref{pro:step3-johnson}, we can conclude the proof as in \Cref{sec:restriction} given an appropriate base case.

The proof of \Cref{pro:step3-johnson} follows the same steps as the proof of \Cref{pro:step3} in \Cref{sec:restriction}. In order to reduce the number of cases, we will replace $0,1,x_i,1-x_i$ by $0^+,0^-,x_i^+,x_i^-$, respectively.

The first step is to construct a ``restriction table''. Given $S \subseteq [n']$ of size $n$ and $T \subset S$ of size $n-1$, we consider the possible structures of $f|_T$ given the structure of $f|_S$:
\[
\begin{array}{c|ll}
f|_S & f|_T \\\hline
0^\pm & 0^\pm \\
x_i^\pm & 0^\pm \text{ if } i \notin T, & x_i^\pm \text{ if } i \in T
\end{array}
\]


The graph on $J(n',n)$ in which we connect two sets $S,S'$ if their intersect has size $n-1$ is known as the \emph{Johnson graph} $G(n',n)$, and it is always connected.

The restriction table implies that if $S \sim S'$ and $f|_S$ is positive (of the form $0^+$ or $x_i^+$) then so are $f|_{S \cap S'}$ and $f|_{S'}$. Since the Johnson graph is connected, we conclude that either all $f|_S$ are positive, or all of them are negative. Without loss of generality, assume that all of them are positive.

If $f|_S = 0$ for all $S \in J(n',n)$ then $f = 0$.

Next, suppose that $f|_{S_0} = x_{i_0}$ for some $S_0$. Consider the subgraph $G_{i_0}$ of $J(n',n)$ consisting of all sets containing ${i_0}$, which is a copy of the Johnson graph $G(n'-1,n-1)$. The restriction table implies that if $S \stackrel{G_i}\sim S'$ and $f|_S = x_{i_0}$ then $f|_{S \cap S'} = x_{i_0}$ and hence also $f|_{S'} = x_{i_0}$. Since $G_{i_0}$ is connected and $f|_{S_0} = x_{i_0}$, we conclude that $f|_S = x_{i_0}$ whenever $i_0 \in S$.

The foregoing implies that $f|_S \in \{0, x_{i_0}\}$ for all $S$. Indeed, suppose that $f|_{S_1} = x_{i_1}$ for some $i_1 \neq i_0$. Then $f|_S = x_{i_1}$ whenever $i_1 \in S$. Since $k \ge 2$, we can find a set $S$ containing both $i_0$ and $i_1$. Then $f|_S = x_{i_0}$ since $i_0 \in S$ and $f|_S = x_{i_1}$ since $i_1 \in S$, and we reach a contradiction.

Since $f|_S = x_{i_0}$ whenever $i_0 \in S$ and $f|_S = 0$ whenever $i_0 \notin S$ (since $f|_S = x_{i_0}$ is impossible in this case), we conclude that $f = x_{i_0}$.

\section{Ramsey theorems}
\label{sec:ramsey}

Our proof employed two Ramsey theorems which we prove here, given the well-known Hales--Jewett theorem.

\subsection{Affine Ramsey from Hales--Jewett}
\label{sec:ramsey-affine}

The proofs of \Cref{lem:ramsey-induction,lem:ramsey-conclusion} employed the following Ramsey theorem.

\begin{proposition} \label{pro:ramsey-affine}
Suppose that the vectors in $\FF_q^n$ are colored using $c$ colors. For every $r$, provided that $n \ge \Nramaff(q,c,r)$, we can find a homogeneous $r$-dimensional affine subspace of $\FF_q^n$.
\end{proposition}

In this appendix we show how to derive \Cref{pro:ramsey-affine} from the Hales--Jewett theorem.

\begin{proposition}[Hales--Jewett] \label{pro:HJ}
Suppose that $[k]^m$ is colored using $c$ colors. Provided that $m \ge \NHJ(k,c)$, we can find a homogeneous \emph{combinatorial line}: a word $w \in ([k] \cup \{*\})^m$ with at least one star such that the $k$ words obtained by replacing all the stars by the same value have the same color.
\end{proposition}

\begin{proof}[Proof of \Cref{pro:ramsey-affine}]
Let $m = \NHJ(q^r,c)$ and $\Nramaff(q,c,r) = mr$. We will prove the result for $n = mr$; it then follows for larger $n$ by considering a copy of $\FF_q^{mr}$ inside $\FF_q^n$.

Given a $c$-coloring of $\FF_q^n$, we construct a $c$-coloring of $[q^r]^m$ in the natural way, by interpreting $[q^r]$ as $\FF_q^r$. \Cref{pro:HJ} thus gives a homogeneous combinatorial line $w$.

Let $w_0$ be the word obtained from $w$ by replacing the stars with zeroes, and for $i \in [r]$ let $w_i$ be obtained from $w$ by (i) replacing all non-stars with zero, (ii) replacing all stars by $e_i$, where $\{e_1,\dots,e_r\}$ is an arbitrary basis of $\FF_q^r$. The $r$-dimensional affine subspace corresponding to $w_0 + \langle w_1,\dots,w_r \rangle$ is homogeneous.
\end{proof}

\subsection{Projective Ramsey from Affine Ramsey}
\label{sec:ramsey-projective}

The proof of \Cref{lem:homogeneous} employed the following Ramsey theorem.

\begin{proposition} \label{pro:ramsey-projective}
Suppose that the points in $\FF_q^n$ are colored using $c$ colors. For every $r$, provided that $n \ge \Nrampro(q,c,r)$, we can find a homogeneous $r$-dimensional subspace of $\FF_q^n$.
\end{proposition}

In this appendix we show how to derive \Cref{pro:ramsey-projective} from \Cref{pro:ramsey-affine}.\footnote{We thank ChatGPT~5.6~Sol for this proof, which specializes the argument of Spencer~\cite{Spencer79}.}
\Cref{pro:ramsey-projective} follows almost directly from the following lemma, which we will prove inductively.

\begin{lemma} \label{lem:ramsey-projective-helper}
Suppose that the points in $\FF_q^n$ are colored using $c$ colors. For every $r$, provided that $n \ge \Nramprohelp(q,c,r)$, we can find $r$ linearly independent vectors $e_1,\dots,e_r \in \FF_q^n$ and $r$ colors $\gamma_1,\dots,\gamma_r$ such that the color of the point $\left\langle\sum_{i=1}^r a_i e_i \right\rangle$ is $\gamma_j$, where $j$ is the smallest index such that $a_j \ne 0$ (note that such an index must exist, and is well-defined).
\end{lemma}
\begin{proof}
The proof is by induction on $r$. When $r = 0$, there is nothing to prove (and so we can take $\Nramprohelp(q,c,0) = 0$). Assuming that the lemma is known to hold for a certain value of $r$, we now prove it for $r + 1$.

Let $m = \Nramprohelp(q,c,r)$ and $\Nramprohelp(q,c,r+1) = \Nramaff(q,c,m) + 1$. Decompose $\FF_q^n$ as the direct sum of a point $\langle e \rangle$ and an $(n-1)$-dimensional subspace $V$. Construct a $c$-coloring $\chi_V$ of the \emph{vectors} in $V$ as follows:
\[
 \chi_V(v) = \chi(\langle e + v \rangle),
\]
where $\chi$ is the given $c$-coloring of the points of $\FF_q^n$. Applying \Cref{pro:ramsey-affine}, we obtain a $\chi_V$-homogeneous $m$-dimensional affine subspace $u + W$ of $V$, say $\chi(\langle e + u + w \rangle) = \gamma_1$ for all $w \in W$. Accordingly, we choose $e_1 = e + u$.
We obtain the remaining vectors $e_2,\dots,e_{r+1}$ and colors $\gamma_2,\dots,\gamma_{r+1}$ by applying the inductive hypothesis to $W$.
\end{proof}

\Cref{pro:ramsey-projective} follows almost directly.
\begin{proof}[Proof of \Cref{pro:ramsey-projective}]
Let $\Nrampro(q,c,r) = \Nramprohelp(q,c,cr)$. Apply \Cref{lem:ramsey-projective-helper} to the given coloring, obtaining $e_1,\dots,e_{cr}$ and $\gamma_1,\dots,\gamma_{cr}$.

By the pigeonhole principle, we can find a subset $I \subseteq [cr]$ of size $r$ such that the colors $\gamma_i$ for $i \in I$ are all equal. Thus the $r$-dimensional subspace generated by $e_i$ for $i \in I$ is homogeneous.
\end{proof}

\bibliographystyle{alphaurl}
\bibliography{biblio}

\end{document}